\documentclass[12pt]{amsart}
\usepackage{amsmath}

\allowdisplaybreaks

\usepackage{amsfonts}

\usepackage{amsthm,amssymb}

\renewcommand{\baselinestretch}{\baselinestretch}
\renewcommand{\baselinestretch}{1.2}

\newtheorem{theorem}{Theorem}[section]
\newtheorem{lemma}[theorem]{Lemma}
\newtheorem{proposition}[theorem]{Proposition}
\newtheorem{corollary}[theorem]{Corollary}
\theoremstyle{definition}

\newtheorem{conjecture}{Conjecture}[section]
\usepackage[all,cmtip]{xy}
\usepackage{enumitem}
\usepackage{tikz-cd}
\usepackage{hyperref}
\usepackage{mathrsfs}
\usepackage{tikz}
\usepackage[english]{babel}
\usepackage{graphicx}
\usepackage{centernot}
\usepackage{mathtools}
\usepackage{cite}
\newlist{inlineroman}{enumerate*}{1}
\setlist[inlineroman]{itemjoin*={{, and }},afterlabel=~,label=\roman*.}

 \numberwithin{dummy}{section}

\newcommand{\tr}{\operatorname{tr}}
\newcommand{\Frob}{\operatorname{Frob}}

\newenvironment{red}{\relax\color{red}}{\relax}
\newenvironment{blue}{\relax\color{blue}}{\hspace*{.5ex}\relax}
\newcommand{\ber}{\begin{red}}
\newcommand{\er}{\end{red}}
\newcommand{\beb}{\begin{blue}}
\newcommand{\eb}{\end{blue}}

\begin{document}
	
\newcommand{\pr}{\partial}
\newcommand{\nl}{\vskip 1pc}
\newcommand{\co}{\mbox{co}}
\newcommand{\ol}{\overline}
\newcommand{\om}{\Omega}
\newcommand{\ra}{\rightarrow}
\newcommand{\epsil}{\varepsilon}
\makeatletter
\newcommand*{\rom}[1]{\expandafter\@slowromancap\romannumeral #1@}
\newcommand{\powerset}{\raisebox{.15\baselineskip}{\Large\ensuremath{\wp}}}
\newcounter{counter}       
\newcommand{\upperRomannumeral}[1]{\setcounter{counter}{#1}\Roman{counter}}
\newcommand{\lowerromannumeral}[1]{\setcounter{counter}{#1}\roman{counter}}
\newcommand*\circled[1]{\tikz[baseline=(char.base)]{
		\node[shape=circle,draw,inner sep=1pt] (char) {#1};}}
\makeatother

\title{On the structure of prime-detecting quasimodular forms in higher levels}

\author{Yeong-Wook Kwon}

\author{Youngmin Lee}

\address{Institute of Basic Science, Korea University, 145 Anam-ro, Seongbuk-gu, Seoul, 02841, Republic of Korea}
\email{ywkwon196884@korea.ac.kr}

\address{Department of Mathematics, Kyonggi University,	154-42, Gwanggyosan-ro, Yeongtong-gu, Suwon-si, Gyeonggi-do, Republic of Korea}

\email{youngminlee@kyonggi.ac.kr}

  \thanks{}
  
\begin{abstract}
   Craig, van Ittersum, and Ono conjectured that every prime-detecting quasimodular form of level $1$ is a quasimodular Eisenstein series. This conjecture was proved by Kane--Krishnamoorthy--Lau and by van Ittersum--Mauth--Ono--Singh independently. However, in higher levels, prime-detecting quasimodular forms need not be Eisenstein. Recently, Kane, Krishnamoorthy, and Lau formulated a natural higher level analogue of the above conjecture and proved it by analytic methods. In a similar direction, but via an alternative approach based on the independence of characters of $\ell$-adic Galois representations, we prove that any prime-detecting quasimodular form on $\Gamma_{0}(N)$ belongs to the direct sum of the spaces of quasimodular Eisenstein series and quasimodular oldforms. 
    Moreover, for a quasimodular form $f$ that is not prime-detecting, we give an upper bound for the number of primes $p$ less than $X$ for which the $p$-th Fourier coefficient of a quasimodular form vanishes.
\end{abstract}

\maketitle
%\thispagestyle{empty}
%\tableofcontents
% \clearpage
\section{Introduction}\label{sec-intro}

\emph{Modular forms} have attracted wide interest because of their connections with many areas of mathematics.
These connections have motivated numerous generalizations of the concept of a modular form. 
One such generalization is a \emph{quasimodular form}. Roughly speaking, quasimodular forms are modular objects that are closed under differentiation and can be described using modular forms together with $E_{2}$ and their derivatives.
A brief review of quasimodular forms will be given in Section \ref{sec-prel}.
Extensive research on quasimodular forms has shown that, like modular forms, they arise across a wide range of mathematical contexts; see, for example, \cite{BO00,BFGW21,CMZ18,I20,IR24,KZ95}.

Recently, Craig, van Ittersum, and Ono studied a special type of quasimodular forms for the full modular group, which they call \emph{prime-detecting quasimodular forms}.
They showed that some prime-detecting quasimodular forms arise from the following partition generating series.
For $a\in\mathbb{Z}_{>0}$, define
$$\mathcal{U}_{a}(q):=\sum_{0<s_{1}<s_{2}<\cdots<s_{a}}\frac{q^{s_{1}+s_{2}+\cdots+s_{a}}}{(1-q^{s_{1}})^{2}(1-q^{s_{2}})^{2}\cdots (1-q^{s_{a}})^{2}}=\sum_{n=1}^{\infty}M_{a}(n)q^{n}.$$
The coefficients $M_{a}(n)$ are called the \emph{MacMahon partition functions}.
They proved that the function
$$(D^{2}-3D+2)\mathcal{U}_{1}(q)-8\mathcal{U}_{2}(q)$$
is a prime-detecting quasimodular form for $\mathrm{SL}_{2}(\mathbb{Z})$. Comparing Fourier coefficients yields the following: an integer $n\geq 2$ is a prime if and only if
\begin{equation*}
    (n^{2}-3n+2)M_{1}(n)-8M_{2}(n)=0.
\end{equation*}
In fact, this is the simplest case of an infinite family of equations involving MacMahon's partition functions whose solutions are precisely the prime numbers; see \cite{CIO24} for details. 
Further prime-detecting identities involving MacMahon partition functions have been established by several authors such as Craig \cite{C24}, Gomez \cite{G25}, and Kang, Matsusaka, and Shin \cite{KMS25}.

Beyond constructing examples, Craig, van Ittersum, and Ono also proposed the following structural conjecture: every prime-detecting quasimodular form for the full modular group is a quasimodular Eisenstein series, i.e., a linear combination of Eisenstein series for $\mathrm{SL}_{2}(\mathbb{Z})$ and their derivatives (see \cite[Conjecture 2]{CIO24}). 
This conjecture was proved by Kane, Krishnamoorthy, and Lau\cite{KKL25} using analytic methods, and van Ittersum, Mauth, Ono and Singh \cite{IMOS25} provided an alternative proof using the theory of $\ell$-adic Galois representations. 

Let $N$ be a positive integer. We call a quasimodular form $f$ on $\Gamma_{0}(N)$ \emph{prime-detecting} if it has a Fourier expansion
$$f(\tau)=\sum_{n=0}^{\infty} a_{f}(n)q^{n}$$
such that, for every integer $n\geq 2$, one has $a_{f}(n)=0$ if and only if $n$ is a prime with $n\nmid N$.
It is important to note that, in the literature, the term “prime-detecting” is sometimes reserved for quasimodular forms with nonnegative Fourier coefficients (e.g. \cite{CIO24}). In this paper, we focus only on the vanishing pattern at primes and do not impose any positivity condition.
Let $\Omega_{N}$ denote the set of all prime-detecting quasimodular forms on $\Gamma_{0}(N)$, and let $\widetilde{M}(N)$ denote the space of quasimodular forms on $\Gamma_{0}(N)$.
As for holomorphic modular forms, $\widetilde{M}(N)$ has a decomposition
$$\widetilde{M}(N)=\widetilde{E}(N)\oplus\widetilde{S}(N),$$
where $\widetilde{E}(N)$ (resp. $\widetilde{S}(N))$ is the space of \emph{quasimodular Eisenstein series} (resp. \emph{quasimodular cusp forms}) on $\Gamma_{0}(N)$. 

In view of the conjecture of Craig, van Ittersum and Ono, it is natural to ask whether $\Omega_{N}\subset\widetilde{E}(N)$ holds.
However, this inclusion is not true in general. In other words, there exists a prime-detecting quasimodular form with nonzero cuspidal part.
We give a concrete example. Given a positive even integer $k$, define
$$G_{k}^{(N)}(\tau):=\sum_{n=1}^{\infty}\left(\sum\limits_{\substack{d\mid n \\ \gcd(n/d,N)=1}}d^{k-1}\right)q^{n}.$$
It is known that for positive odd integers $l>k$, the function
$$f_{k,l}^{(N)}(\tau):=(D^{l}+1)G_{k+1}^{(N)}(\tau)-(D^{k}+1)G_{l+1}^{(N)}(\tau)$$
belongs to $\widetilde{E}(N)$ and is prime-detecting. Here $D=\frac{1}{2\pi i}\frac{d}{d\tau}$; see \cite[Sections 3.2 and 4.2]{KMS25}. Fix positive odd integers $l>k$. Then, for any divisor $d$ of $N$ with $d>1$, $f_{k,l}^{(N)}(\tau)+\Delta(d\tau)\in\Omega_{N}$ and its cuspidal part is $\Delta(d\tau)\neq 0$, where $\Delta(\tau)$ is the modular discriminant.

As will be explained in Section \ref{sec-prel}, the cuspidal space $\widetilde{S}(N)$ further decomposes as
$$\widetilde{S}(N)=\widetilde{S}^{\mathrm{new}}(N)\oplus \widetilde{S}^{\mathrm{old}}(N),$$
where $\widetilde{S}^{\mathrm{new}}(N)$ (resp. $\widetilde{S}^{\mathrm{old}}(N)$) is the space of \emph{quasimodular newforms} (resp. \emph{quasimodular oldforms}) on $\Gamma_{0}(N)$.
One can check that, in the above example, the cuspidal part of $f_{k,l}^{(N)}(\tau)+\Delta(d\tau)$ lies in $\widetilde{S}^{\mathrm{old}}(N)$.
Thus, one may expect that the failure of the inclusion $\Omega_{N}\subset \widetilde{E}(N)$ is caused by contributions from quasimodular oldforms.
This suggests the following refinement:
\begin{conjecture}\label{conj-newvanish}
    Let $N$ be a positive integer. Then
    \[ \Omega_{N}=\Omega_{N}\cap \left(\widetilde{E}(N) \oplus \widetilde{S}^{\mathrm{old}}(N)\right). \]
\end{conjecture}

For a positive integer $N$, let
$$\widetilde{\Omega}_{N}:=\left\{f=\sum_{n=0}^{\infty}a_{f}(n)q^{n}\in\widetilde{M}(N):a_{f}(p)=0~\text{for every prime}~p\nmid N\right\}.$$
Our first main result shows that a slightly stronger result than Conjecture \ref{conj-newvanish} holds.

\begin{theorem}\label{thm1}
    Let $N$ be a positive integer. Then, 
    \[ \widetilde{\Omega}_{N}=\widetilde{\Omega}_{N}\cap \left(\widetilde{E}(N) \oplus \widetilde{S}^{\mathrm{old}}(N)\right). \]
\end{theorem}

The following corollary is an immediate consequence of Theorem \ref{thm1}.

\begin{corollary}\label{cor-conjtrue}
    Conjecture \ref{conj-newvanish} is true.
\end{corollary}

The proof of Theorem \ref{thm1} is based on $\ell$-adic Galois representations, especially the linear independence of characters of $G_{\mathbb{Q}}$. It can be summarized as follows: Starting with $f\in\Omega_{N}$, we write
$$f=f_{E}+f_{S}^{\mathrm{new}}+f_{S}^{\mathrm{old}}$$
using the decomposition
$$\widetilde{M}(N)=\widetilde{E}(N)\oplus\widetilde{S}^{\mathrm{new}}(N)\oplus\widetilde{S}^{\mathrm{old}}(N).$$
To prove Theorem \ref{thm1}, it suffices to show that $f_{S}^{\mathrm{new}}=0$.
The key idea is to interpret, after certain reduction process, the relations $a_{f}(p)=0$ (for $p\nmid N\ell$) as a linear combination of traces of pairwise nonisomorphic representations of $G_{\mathbb{Q}}=\mathrm{Gal}(\overline{\mathbb{Q}}/\mathbb{Q})$ at $\mathrm{Frob}_{p}$, and then to use the linear independence of characters.

In \cite{KKL25-2}, Kane, Krishnamoorthy and Lau proved a result in a broader context, which specializes to Theorem \ref{thm1}; for details, see \cite[Theorem 1.4]{KKL25-2}. 
Their approach is essentially analytic, whereas our proof of Theorem \ref{thm1} is based on the theory of Galois representations, as mentioned above.

Recall that a prime-detecting quasimodular form $f$ satisfies the vanishing condition $a_f(p)=0$ for every prime $p\nmid N$. 
The following theorem shows that this is an exceptional case. 
Once this vanishing fails at a single prime $p_0$ with $p_0\nmid N$, nonvanishing occurs for a set of primes of positive density. 
Under additional hypotheses on $f$, one also obtains a quantitative upper bound for the number of primes $p\leq X$ such that $a_f(p)=0$. 

\begin{theorem}\label{thm: quan-1}
Let $f$ be a quasimodular form on $\Gamma_0(N)$.
Assume that there exists a prime $p_0\nmid N$ such that $a_f(p_0)\neq 0$.
Then, the following holds. 

\begin{enumerate}

    \item[\textnormal{(i)}] The set of primes $p$ for which  
    \[ a_f(p)\neq 0 \]
has a positive density.

    \item[\textnormal{(ii)}] Assume that $f=\sum_{i\in I}c_i D^{m_i}f_i$,
    with $I$ finite, $c_i\in\mathbb{C}^{\times}$, and $m_i\in\mathbb{Z}_{\ge 0}$, where each $f_i$ is either an Eisenstein series
or a non-CM Hecke eigenform. Assume that at least one $f_i$ is a non-CM Hecke eigenform and that for $i\neq j$,
the form $f_i$ is not Galois conjugate to any Dirichlet twist of $f_j$. Then, there exists $\delta>0$ such that
\[ \# \{ p\leq X : a_f(p)=0\} = O\left(\frac{X/\log X}{\epsilon(X)^{\delta}}\right), \]
    where 
    \[\epsilon(X):=(\log X)(\log \log X)^{-2}(\log \log \log X)^{-1}. \]

\end{enumerate}
\end{theorem}

The remainder of this paper is organized as follows. In Section \ref{sec-prel}, we review basic notions and properties of quasimodular forms, and recall the theory of Galois representations attached to modular forms. 
In Section \ref{sec-proof1}, we prove Theorem \ref{thm1}, and in Section \ref{sec-proof2} we prove Theorem \ref{thm: quan-1}.

\section{Preliminaries}\label{sec-prel}

In this section, we briefly review the basic definitions and properties of modular forms and quasimodular forms. For details, we refer the reader to \cite{CS17,DS05,KZ95,Z08}. Let $N$ be a positive integer and let $k$ be an integer. The group $\Gamma_{0}(N)$ is a subgroup of $\mathrm{SL}_{2}(\mathbb{Z})$ defined by
$$\Gamma_{0}(N):=\left\{\begin{pmatrix}
    a & b \\ c & d
\end{pmatrix}\in\mathrm{SL}_{2}(\mathbb{Z}): N\mid c\right\}.$$
A holomorphic function $f$ on $\mathbb{H}$ is a \textit{modular form} of weight $k$ for $\Gamma_{0}(N)$ if it satisfies
\begin{equation}\label{eq-modtr}
    f\left(\frac{a\tau+b}{c\tau+d}\right)=(c\tau+d)^{k}f(\tau)
\end{equation}
for all $\left(\begin{smallmatrix}
        a & b \\ c & d
\end{smallmatrix}\right)\in\Gamma_{0}(N)$ and has moderate growth at each cusp of $\Gamma_{0}(N)$.
Denote by $M_{k}(N)$ the space of modular forms of weight $k$ for $\Gamma_{0}(N)$. This space decomposes as
\begin{equation}\label{eq-moddecomp}
    M_{k}(N)=E_{k}(N)\oplus S_{k}(N),
\end{equation}
where $E_{k}(N)$ is the Eisenstein subspace of $M_{k}(N)$ and $S_{k}(N)$ is the subspace of cusp forms.

For an even integer $k>2$ and a positive integer $N$, let $A_{N,k}$ be the set of pairs $(\varphi,t)$, where $\varphi$ is a primitive Dirichlet character modulo $u$, and $t$ is a positive integer such that $tu^{2}\mid N$. For each $(\varphi,t)\in A_{N,k}$, define
$$E_{k}^{\varphi}(\tau):=\delta_{\varphi,\mathbf{1}_{1}}L(1-k,\varphi)+2\sum_{n=1}^{\infty}\sigma_{k-1}^{\varphi}(n)q^{n},$$
where
$$\sigma_{k-1}^{\varphi}(n)=\sum_{d\mid n}\varphi(d)\overline{\varphi(n/d)}d^{k-1}.$$

\begin{proposition}\label{prop-meisb}
    Let $k>2$ be an even integer, and let $N$ be a positive integer. Then the set
    \begin{equation*}
        \begin{split}
            \mathcal{A}_{k}(N):=\{E_{k}^{\varphi,t}(\tau):=E_{k}^{\varphi}(t\tau):(\varphi,t)\in A_{N,k}\}
        \end{split}
    \end{equation*}
    is a basis for $E_{k}(N)$.
\end{proposition}

\begin{proof}
    See \cite[p. 129]{DS05}.
\end{proof}

For a positive integer $N$, let $A_{N,2}$ be the set of pairs $(\varphi,t)$, where $\varphi$ is a primitive Dirichlet character modulo $u$, and $t$ is a positive integer such that $tu^{2}\mid N$. For any pair $(\varphi,t)\in A_{N,2}$, define
\begin{equation*}
    E_{2}^{\varphi,t}(\tau):=
    \begin{cases}
        E_{2}(\tau)-tE_{2}(t\tau) & \text{if}~\varphi=\mathbf{1}_{1},\\
        E_{2}^{\varphi}(t\tau) & \text{otherwise},
    \end{cases}
\end{equation*}
where
\begin{align*}
    E_{2}(\tau)&:=1-24\sum_{n=1}^{\infty}\left(\sum_{d\mid n}d\right)q^{n},\\
    E_{2}^{\varphi}(\tau)&:=\delta_{\varphi,\mathbf{1}_{1}}L(-1,\varphi)+2\sum_{n=1}^{\infty}\sigma_{1}^{\varphi}(n)q^{n},\qquad\sigma_{1}^{\varphi}(n):=\sum_{d\mid n}\varphi(d)\overline{\varphi(n/d)}d.
\end{align*}

\begin{proposition}\label{prop-meisbwttwo}
    Let $N$ be a positive integer. Then the set
    $$\mathcal{A}_{2}(N):=\{E_{2}^{\varphi,t}:(\varphi,t)\in A_{N,2}\}$$
    is a basis for $E_{2}(N)$.
\end{proposition}

\begin{proof}
    See \cite[Section 4.6]{DS05}.
\end{proof}

For a positive even integer $k$ and a positive integer $N$, denote by $\mathcal{N}_{k}(N)$ the set of all newforms in $S_{k}(N)$.

\begin{proposition}\label{prop-mcuspb}
    Let $k$ be a positive even integer and let $N$ be a positive integer. Then the set
    $$\mathcal{B}_{k}(N):=\{f(n\tau): n,L\in\mathbb{Z}_{>0},~nL\mid N~\text{and}~f\in\mathcal{N}_{k}(L)\}$$
    is a basis for $S_{k}(N)$.
\end{proposition}

\begin{proof}
    See \cite[pp. 197--198]{DS05}.
\end{proof}

In the following proposition, we recall the $\ell$-adic Galois representation attached to a normalized Hecke eigenform.

\begin{proposition}\label{prop-cuspgalrep}
    Let $k\geq 2$ be an even integer, and let $f=\sum_{n=1}^{\infty}a_{f}(n)q^{n}\in S_{k}(N)$ be a normalized eigenform of the Hecke operator $T_{p}$ for each prime $p\nmid N$. Moreover, let $K$ be a number field such that $a_{f}(n)\in K$ for every positive integer $n$, let $\ell$ be a prime with $\ell\nmid N$, and let $\lambda$ be a finite place of $K$ lying above $\ell$. Then there exists a representation
    $$\rho_{f,\lambda}:G_{\mathbb{Q}}\rightarrow\mathrm{GL}_{2}(K_{\lambda})$$
    such that
    \begin{enumerate}
        \item[\textnormal{(i)}] $\rho_{f,\lambda}$ is irreducible,
        \item[\textnormal{(ii)}] $\rho_{f,\lambda}$ is odd,
        \item[\textnormal{(iii)}] for every prime $p\nmid N\ell$, the representation $\rho_{f,\lambda}$ is unramified at $p$ and satisfies
        $$\mathrm{tr}(\rho_{f,\lambda}(\mathrm{Frob}_{p}))=a_{f}(p)\quad\text{and}\quad\det(\rho_{f,\lambda}(\mathrm{Frob}_{p}))=p^{k-1}.$$
    \end{enumerate}
\end{proposition}
\begin{proof}
    In the case $k=2$, this proposition follows from the work of Eichler \cite{E54}, Shimura \cite{S58}, and Igusa \cite{I59}. Deligne \cite{D71} later generalized their construction to cusp forms of weight $k\geq 2$.    
\end{proof}

Let
\begin{equation*}
    M(N):=\bigoplus_{k\in\mathbb{Z}}M_{k}(N),\quad E(N):=\bigoplus_{k\in\mathbb{Z}}E_{k}(N),\quad S(N):=\bigoplus_{k\in\mathbb{Z}}S_{k}(N).
\end{equation*}
Then
$$M(N)=E(N)\oplus S(N).$$

We now consider a wider class of functions. An \textit{almost-holomorphic modular form} of weight $k$ for $\Gamma_{0}(N)$ is a smooth function $F$ on $\mathbb{H}$ satisfying the following two conditions:
\begin{enumerate}
    \item[(i)] $F\left(\frac{a\tau+b}{c\tau+d}\right)=(c\tau+d)^{k}F(\tau)$ for every $\left(\begin{smallmatrix}
        a & b \\ c & d
    \end{smallmatrix}\right)\in\Gamma_{0}(N)$,

    \item[(ii)] $$F(\tau)=\sum_{i=0}^{m}f_{i}(\tau)Y^{-i}$$
    for some holomorphic functions $f_{0},f_{1},\ldots,f_{m}$ on $\mathbb{H}$ with moderate growth at each cusp of $\Gamma_{0}(N)$.
\end{enumerate}
We say that a function $f$ is a \textit{quasimodular form} of weight $k$ for $\Gamma_{0}(N)$ if there is an almost-holomorphic modular form $F$ of weight $k$ for $\Gamma_{0}(N)$ whose constant term (when $F$ is expressed as a polynomial in $1/Y$) is $f$. Let us denote the space of quasimodular forms of weight $k$ for $\Gamma_{0}(N)$ by $\widetilde{M}_{k}(N)$, and we set
$$\widetilde{M}(N):=\bigoplus_{k\in\mathbb{Z}}\widetilde{M}_{k}(N).$$
Hereafter, given any quasimodular form $f$ on $\Gamma_{0}(N)$ and any nonnegative integer $n$, we shall denote by $a_{f}(n)$ the coefficient of $q^{n}$ in the Fourier expansion of $f$. Hence,
$$f(\tau)=\sum_{n=0}^{\infty}a_{f}(n)q^{n}.$$

Define
$$D:=\frac{1}{2\pi i}\frac{d}{d\tau}=q\frac{d}{dq}.$$
The following result describes the structure of quasimodular forms.

\begin{proposition}\label{prop-qmfstr}
    Let $N$ be a positive integer. Then we have
    \begin{enumerate}
        \item[\textnormal{(i)}] $\widetilde{M}(N)$ is closed under the operator $D$.
        
        \item[\textnormal{(ii)}] $\widetilde{M}(N)=M(N)\otimes\mathbb{C}[E_{2}]$, i.e., every quasimodular form for $\Gamma_{0}(N)$ can be written uniquely as a polynomial in $E_{2}$ with coefficients in $M(N)$.
        
        \item[\textnormal{(iii)}] For an even positive integer $k$, we have
        $$\widetilde{M}_{k}(N)=\bigoplus_{i=0}^{k/2-1}D^{i}M_{k-2i}(N)\oplus\mathbb{C}\cdot D^{k/2-1}E_{2}.$$
    \end{enumerate}
\end{proposition}

\begin{proof}
    See \cite[p. 3]{KZ95} and \cite[p. 59]{Z08}.
\end{proof}

By the definition of a quasimodular form and Proposition \ref{prop-qmfstr} (ii),
\begin{equation*}
    \widetilde{M}_{k}(N)=
    \begin{cases}
        0 & \text{if}~k~\text{is negative or odd},\\
        \mathbb{C} & \text{if}~k=0.
    \end{cases}
\end{equation*}
Given a nonnegative integer $k$, put
\begin{align}
    \widetilde{E}_{2k}(N)&:=
    \begin{cases}
        \mathbb{C}\cdot 1 & \text{if}~k=0,\\
        \displaystyle\left(\bigoplus_{i=0}^{k-1}D^{i}E_{2k-2i}(N)\right)\oplus\mathbb{C}\cdot D^{k-1}E_{2} & \text{if}~k>0,
    \end{cases}\label{eq-qmfeis}\\
    \widetilde{S}_{2k}(N)&:=
    \begin{cases}
        0 & \text{if}~k=0,\\
        \displaystyle\bigoplus_{i=0}^{k-1}D^{i}S_{2k-2i}(N) & \text{if}~k>0.
    \end{cases}\nonumber
\end{align}
For each positive integer $k$, we obtain from Propositions \ref{prop-meisb}, \ref{prop-meisbwttwo} and \ref{prop-mcuspb} the following descriptions of $\widetilde{E}_{2k}(N)$ and $\widetilde{S}_{2k}(N)$:
\begin{align}
    \widetilde{E}_{2k}(N)&=\left(\bigoplus_{i=1}^{k}\bigoplus_{(\varphi,t)\in A_{N,2i}}\mathbb{C}\cdot D^{k-i}E_{2i}^{\varphi,t}\right)\oplus\mathbb{C}\cdot D^{k-1}E_{2}\label{eq-qmfeisbasis}
\end{align}
and
\begin{align*}
    \widetilde{S}_{2k}&(N)\\
    &=\left(\bigoplus_{i=1}^{k}\bigoplus_{L\mid N}\bigoplus_{f\in\mathcal{N}_{2i}(L)}\mathbb{C}\cdot (D^{k-i}f)(\tau)\right)\oplus\left(\bigoplus_{i=1}^{k}\bigoplus_{L\mid N}\bigoplus_{f\in\mathcal{N}_{2i}(L)}\bigoplus_{1<n\mid(N/L)}\mathbb{C}\cdot D^{k-i}(f(n\tau))\right)\\
    &=\widetilde{S}_{2k}^{\mathrm{new}}(N)\oplus\widetilde{S}_{2k}^{\mathrm{old}}(N),
\end{align*}
where
\begin{align}
    \widetilde{S}_{2k}^{\mathrm{new}}(N)&:=\bigoplus_{i=1}^{k}\bigoplus_{L\mid N}\bigoplus_{f\in\mathcal{N}_{2i}(L)}\mathbb{C}\cdot (D^{k-i}f)(\tau),\label{eq-qmfnew}\\
    \widetilde{S}_{2k}^{\mathrm{old}}(N)&:=\bigoplus_{i=1}^{k}\bigoplus_{L\mid N}\bigoplus_{f\in\mathcal{N}_{2i}(L)}\bigoplus_{1<n\mid(N/L)}\mathbb{C}\cdot D^{k-i}(f(n\tau)).\nonumber
\end{align}
Furthermore, define
\begin{align}
    \widetilde{E}(N)&:=\bigoplus_{k\in\mathbb{Z}}\widetilde{E}_{k}(N)=\bigoplus_{k\in\mathbb{Z}_{\geq 0}}\widetilde{E}_{2k}(N),\label{eq-qmfeistotal}\\
    \widetilde{S}(N)&:=\bigoplus_{k\in\mathbb{Z}}\widetilde{S}_{k}(N)=\bigoplus_{k\in\mathbb{Z}_{>0}}\widetilde{S}_{2k}(N),\nonumber\\
    \widetilde{S}^{\mathrm{new}}(N)&:=\bigoplus_{k\in\mathbb{Z}_{>0}}\widetilde{S}_{2k}^{\mathrm{new}}(N),\label{eq-qmfnewtotal}\\
    \widetilde{S}^{\mathrm{old}}(N)&:=\bigoplus_{k\in\mathbb{Z}_{>0}}\widetilde{S}_{2k}^{\mathrm{old}}(N).\nonumber
\end{align}
Then
$$\widetilde{S}(N)=\widetilde{S}^{\mathrm{new}}(N)\oplus\widetilde{S}^{\mathrm{old}}(N).$$
We call elements of $\widetilde{E}(N)$ (resp. $\widetilde{S}(N)$) \textit{quasimodular Eisenstein series} (resp. \textit{quasimodular cusp forms}) for $\Gamma_{0}(N)$. Moreover, we call elements of $\widetilde{S}^{\mathrm{new}}(N)$ (resp. $\widetilde{S}^{\mathrm{old}}(N)$) \textit{quasimodular newforms} (resp. \textit{quasimodular oldforms}) for $\Gamma_{0}(N)$. Using \eqref{eq-moddecomp} and Proposition \ref{prop-qmfstr} (iii), we see that
$$\widetilde{M}(N)=\widetilde{E}(N)\oplus\widetilde{S}(N)=\widetilde{E}(N)\oplus\widetilde{S}^{\mathrm{new}}(N)\oplus\widetilde{S}^{\mathrm{old}}(N).$$

\section{Proof of Theorem \ref{thm1}}\label{sec-proof1}

In this section, we prove Theorem \ref{thm1}. To this end, we first establish several lemmas.

\begin{lemma}\label{lem-qmfalgbasis}
    Let $N$ be a positive integer. For each $k\geq 0$, the space $\widetilde{M}_{2k}(N)$ admits a basis consisting of forms whose Fourier coefficients lie in $\mathbb{Q}(\zeta_{N})$.
\end{lemma}

\begin{proof}
For $k=0$, we have $\widetilde{M}_{0}(N)=\mathbb{C}\cdot 1$, and the constant function $1$ has rational Fourier coefficients. Thus, the claim holds for $k=0$.

Now assume that $k\geq 1$. For $i\in\{1,2,\ldots,k\}$, we have 
$$M_{2i}(N)=E_{2i}(N)\oplus S_{2i}(N).$$ 
By Propositions~\ref{prop-meisb} and \ref{prop-meisbwttwo}, the space $E_{2i}(N)$
admits a basis consisting of forms with Fourier coefficients in $\mathbb{Q}(\zeta_{N})$. On the other hand, by \cite[Theorem 3.52]{S71}, the space $S_{2i}(N)$ admits a basis consisting of forms whose Fourier coefficients are rational. Consequently, $M_{2i}(N)$ admits a basis consisting of forms with Fourier coefficients in $\mathbb{Q}(\zeta_{N})$. 

By Proposition \ref{prop-qmfstr}(iii),
\begin{equation*}
    \widetilde{M}_{2k}(N)=\left(\bigoplus_{i=1}^{k} D^{k-i} M_{2i}(N)\right)\oplus \mathbb{C}\cdot D^{k-1}E_{2}.
\end{equation*}
Since the operator $D=q\frac{d}{dq}$ preserves $\mathbb{Q}(\zeta_{N})[\![q]\!]$, we conclude that the space $\widetilde{M}_{2k}(N)$ admits a basis consisting of forms whose Fourier coefficients lie in $\mathbb{Q}(\zeta_{N})$.
\end{proof}

Recall that we defined in Section \ref{sec-intro} the set $\widetilde{\Omega}_{N}$ by
$$\widetilde{\Omega}_{N}:=\left\{f=\sum_{n=0}^{\infty}a_{f}(n)q^{n}\in\widetilde{M}(N):a_{f}(p)=0~\text{for every prime}~p\nmid N\right\}.$$
The following lemma shows that every element of $\widetilde{\Omega}_{N}$ is a $\mathbb{C}$-linear combination of elements in the same set whose Fourier coefficients lie in $\mathbb{Q}(\zeta_{N})$.

\begin{lemma}\label{lem-reduction}
    Let $N$ be a positive integer and let $f(\tau)=\sum_{n=0}^{\infty}a_{f}(n)q^{n}\in\widetilde{M}(N)$.
    Assume that $a_{f}(p)=0$ for all primes $p\nmid N$.
    Then there exist $f_{1},\dots,f_{d}\in \widetilde{M}(N)\cap \mathbb{Q}(\zeta_{N})[\![q]\!]$ and $\beta_{1},\ldots,\beta_{d}\in\mathbb{C}$ such that
    $$f=\sum_{j=1}^{d}\beta_{j}f_{j},$$
    and, for each $j\in\{1,\ldots,d\}$ and each prime $p\nmid N$, the $p$-th Fourier coefficient of $f_{j}$ is zero.
\end{lemma}

\begin{proof}
    If $f=0$, take $d=1$, $f_{1}=1$ and $\beta_{1}=0$. 
    Now we assume that $f\neq 0$. 
    Then there is $k_{f}\in\mathbb{Z}_{\geq 0}$ such that $f\in\bigoplus_{k=0}^{k_{f}}\widetilde{M}_{2k}(N)$. By Lemma \ref{lem-qmfalgbasis}, the space $\bigoplus_{k=0}^{k_{f}}\widetilde{M}_{2k}(N)$ has a basis $\{h_{1},\ldots,h_{m}\}$ with $h_{j}\in\mathbb{Q}(\zeta_{N})[\![q]\!]$ for $j=1,\ldots,m$.  
    Write  
    $$f=\sum_{i=1}^{m}\alpha_{i} h_{i}.$$
    
    Let $W$ be the $\mathbb{Q}(\zeta_{N})$-vector space spanned by $\alpha_{1},\ldots,\alpha_{m}$, and choose a $\mathbb{Q}(\zeta_{N})$-basis $\beta_{1},\ldots,\beta_{d}$ for $W$.
    Then, for $i\in\{1,2,\ldots,m\}$, there exist $\alpha_{i,1},\ldots,\alpha_{i,d}\in\mathbb{Q}(\zeta_{N})$ such that $\alpha_{i}=\sum_{j=1}^{d}\alpha_{i,j}\beta_{j}$.
    Define, for $1\leq j\leq d$, 
    $$f_{j}:=\sum_{i=1}^{m} \alpha_{i,j}h_{i}=\sum_{n=0}^{\infty}a_{f_{j}}(n)q^{n}\in \mathbb{Q}(\zeta_{N})[[q]].$$
    Then $f_{j}\in\bigoplus_{k=0}^{k_{f}}\widetilde{M}_{2k}(N)\subset\widetilde{M}(N)$ for $j=1,2,\ldots,d$, and
    \begin{align*}
        f=\sum_{i=1}^{m}\alpha_{i}h_{i}=\sum_{i=1}^{m}\sum_{j=1}^{d}\alpha_{i,j}\beta_{j}h_{i}=\sum_{j=1}^{d}\beta_{j}\sum_{i=1}^{m}\alpha_{i,j}h_{i}=\sum_{j=1}^{d}\beta_{j}f_{j}.
    \end{align*}
    For each prime $p\nmid N$, we have 
    $$0=a_{f}(p)=\sum_{j=1}^{d} \beta_{j}a_{f_{j}}(p).$$
    Since $\{\beta_{j}\}$ is $\mathbb{Q}(\zeta_N)$-linearly independent, we get 
    $$a_{f_{j}}(p)=0$$
    for all $j\in\{1,2,\ldots,d\}$ and all primes $p\nmid N$.
\end{proof}

The next lemma is a straightforward linear-algebraic consequence, but we provide a proof for completeness. 

\begin{lemma}\label{lem-coefffield}
    Let $N$ be a positive integer, let $K$ be a subfield of $\mathbb{C}$, and let $h_{1},\dots,h_{r}\in\widetilde{M}(N)\cap K[\![q]\!]$. Suppose that the $h_{j}$ are $K$-linearly independent.
    If $\alpha_{1},\ldots,\alpha_{r}\in\mathbb{C}$ satisfy $\sum_{j=1}^{r}\alpha_{j}h_{j}\in K[\![q]\!]$, then $\alpha_{j}\in K$ for all $j$.
\end{lemma}

\begin{proof}
    For convenience, let $f=\sum_{j=1}^{r}\alpha_{j}h_{j}$. Write 
    $$f(\tau)=\sum_{n=0}^{\infty}a_{f}(n)q^{n},\quad h_{j}(\tau)=\sum_{n=0}^{\infty} a_{h_{j}}(n)q^{n}\quad (j=1,\ldots,r).$$
    For each $n\geq 0$, define
    $$w(n):=(a_{h_{1}}(n),\ldots,a_{h_{r}}(n))\in K^{r},$$
    and let $V$ be the $K$-subspace of $K^{r}$ spanned by $\{w(n):n\geq 0\}$.

    We claim that $V=K^{r}$. 
    Indeed, note that the standard symmetric bilinear form on $K^{r}$ defined by
    $$(v_{1},\ldots,v_{r})\cdot (w_{1},\ldots,w_{r}):=v_{1}w_{1}+\cdots+v_{r}w_{r}$$
    is nondegenerate.
    Thus, $\dim_{K}(V)+\dim_{K}(V^{\perp})=\dim_{K}(K^{r})$, where $V^{\perp}:=\{v\in K^{r}:v\cdot w=0~\text{for all}~w\in V\}$.
    Let $v=(v_{1},\ldots,v_{r})$ be any vector in $V^{\perp}$.
    Since $v\cdot(a_{h_{1}}(n),\ldots,a_{h_{r}}(n))=0$ for all $n\geq 0$, we have
    $$\sum_{j=1}^{r}v_{j}h_{j}=\sum_{n=0}^{\infty}\left(\sum_{j=1}^{r}v_{j}a_{h_{j}}(n)\right)q^{n}=0.$$
    By the $K$-linear independence of the $h_{j}$, we obtain $v=0$. 
    Thus, $V^{\perp}=\{0\}$, and hence $\dim_{K}(V)=\dim_{K}(K^{r})$. This proves the claim.
    
    Since $V=K^{r}$, there exist nonnegative integers $n_{1},\ldots,n_{r}$ such that $\{w(n_{1}),\ldots,w(n_{r})\}$ is a $K$-basis for $K^{r}$.
    Let
    $$A:=(a_{h_{j}}(n_{i}))_{1\leq i,j\leq r}.$$
    Then $A\in\mathrm{GL}_{r}(K)$, because $\{w(n_{1}),\ldots,w(n_{r})\}$ is a basis for $K^{r}$.
    Moreover, from the equality $f=\sum_{j=1}^{r}\alpha_{j}h_{j}$, we obtain
    $$(a_{f}(n_{1}),\ldots,a_{f}(n_{r}))=(\alpha_{1},\ldots,\alpha_{r})A.$$
    Since $f\in K[\![q]\!]$, we have $a_{f}(n_{i})\in K$ for $i=1,\ldots,r$, and thus
    $$(\alpha_{1},\ldots,\alpha_{r})=(a_{f}(n_{1}),\ldots,a_{f}(n_{r}))A^{-1}\in K^{r}.$$
\end{proof}

In the following lemma, we show that every quasimodular oldform has the vanishing $p$-th Fourier coefficient for primes $p\nmid N$. 

\begin{lemma}\label{lem-quasiold}
    Let $N$ be a positive integer. If $p$ is a prime with $p\nmid N$, then the $p$-th Fourier coefficient of every element of $\widetilde{S}^{\mathrm{old}}(N)$ is zero.
\end{lemma}

\begin{proof}
    Let $L\mid N$, let $f(\tau)=\sum_{m=1}^{\infty}a_{f}(m)q^{m}$ be a holomorphic cusp form on $\Gamma_{0}(L)$, let $n\geq 2$ be a divisor of $N/L$, and set $g(\tau):=f(n\tau)$. If $r$ is a nonnegative integer, then
    $$(D^{r}g)(\tau)=\sum_{m=1}^{\infty}(mn)^{r}a_{f}(m)q^{mn}.$$
    Since $p\nmid N$ and $1<n\mid N$, $p$ cannot be a multiple of $n$. Thus, $p$-th Fourier coefficient of $D^{r}g$ is zero.

    Note that the set
    $$\bigcup_{L\mid N}\{D^{k-i}(f(n\tau)):k\in\mathbb{Z}_{>0},~1\leq i\leq k,~f\in\mathcal{N}_{2i}(L),~1<n\mid(N/L)\}$$
    spans the space $\widetilde{S}^{\mathrm{old}}(N)$. By the above discussion, the $p$-th Fourier coefficient of each spanning element is zero, and therefore the $p$-th Fourier coefficient of every element of $\widetilde{S}^{\mathrm{old}}(N)$ is zero.
\end{proof}

Next, we establish an analogous result for certain quasimodular Eisenstein series.

\begin{lemma}\label{lem-quasioldeis}
    Let $N$ be a positive integer, let $k\geq 2$ be an even integer, let $(\varphi,t)\in A_{N,k}$, let $r\in\mathbb{Z}_{\geq 0}$ and let $p$ be a prime with $p\nmid N$.
    \begin{enumerate}
        \item[\textnormal{(i)}] If $(k,\varphi)\neq (2,\mathbf{1}_{1})$ and $t>1$, then the $p$-th Fourier coefficient of $D^{r}E_{k}^{\varphi,t}$ is $0$.

        \item[\textnormal{(ii)}] If $(k,\varphi)=(2,\mathbf{1}_{1})$ and $t>1$, then the $p$-th Fourier coefficient of $D^{r}E_{k}^{\varphi,t}=D^{r}E_{2}^{\mathbf{1}_1,t}$ equals the $p$-th Fourier coefficient of $D^{r}E_{2}$.
    \end{enumerate}
\end{lemma}

\begin{proof}
    (i) For convenience, set $f(\tau)=E_{k}^{\varphi}(\tau)$. Write $f(\tau)=\sum_{n=0}^{\infty}a_{f}(n)q^{n}$. Since $(k,\varphi)\neq(2,\mathbf{1}_{1})$, $E_{k}^{\varphi,t}(\tau)=E_{k}^{\varphi}(t\tau)=f(t\tau)=\sum_{n=0}^{\infty}a_{f}(n)q^{tn}$. It follows that
    $$(D^{r}E_{k}^{\varphi,t})(\tau)=\sum_{n=0}^{\infty}(tn)^{r}a_{f}(n)q^{tn}.$$
    Since $1<t\mid N$ and $p\nmid N$, $p$ is not a multiple of $t$. Hence, the coefficient of $q^{p}$ in the Fourier expansion of $D^{r}E_{k}^{\varphi,t}$ is zero.

    (ii) By definition, $E_{2}^{\mathbf{1}_{1},t}(\tau)=E_{2}(\tau)-tE_{2}(t\tau)$. Thus,
    $$(D^{r}E_{2}^{\mathbf{1}_{1},t})(\tau)=(D^{r}E_{2})(\tau)-t(D^{r}E_{2}(t\tau)).$$
    Arguing as in the previous case, we see that the $p$-th Fourier coefficient of $(D^{r}E_{2}(t\tau))$ is zero. Consequently, the $p$-th Fourier coefficient of $D^{r}E_{k}^{\varphi,t}=D^{r}E_{2}^{\mathbf{1}_1,t}$ is equal to that of $D^{r}E_{2}$.
\end{proof}

Let $N$ be a positive integer. Let $g=\sum_{n=1}^{\infty}a_{g}(n)q^{n}$ (resp. $g'=\sum_{n=1}^{\infty}a_{g'}(n)q^{n}$) be a cusp form of weight $k$ (resp. $k'$) on $\Gamma_{0}(N)$. Assume that $g\in\mathcal{N}_{k}(L)$ and $g'\in\mathcal{N}_{k'}(L')$ for some $L,L'\mid N$. Then there exists a number field $K$ such that $a_{g}(n), a_{g'}(n)\in K$ for all $n\in\mathbb{Z}_{>0}$. Fix a prime $\ell\nmid N$ and a finite place $\lambda$ of $K$ lying above $\ell$. Let $\rho_{g,\lambda}$ (resp. $\rho_{g',\lambda}$) be the Galois representation attached to $g$ (resp. $g'$) as in Proposition \ref{prop-cuspgalrep}.

In the next lemma, we compare two Galois representations twisted by powers of cyclotomic characters. 

\begin{lemma}\label{lem-rigidity}
    Let $m,m'\in\mathbb{Z}_{\geq 0}$. If $$\rho_{g,\lambda}\otimes\chi_{\ell}^{m}\cong\rho_{g',\lambda}\otimes\chi_{\ell}^{m'},$$ 
    then $g=g'$ and $m=m'$. Here, $\chi_{\ell}$ denotes the $\ell$-adic cyclotomic character.
\end{lemma}

\begin{proof}
    Since $\chi_{\ell}(\mathrm{Frob}_{p})=p$ for all primes $p$ with $p\nmid \ell N$ and $\rho_{g,\lambda}\otimes\chi_{\ell}^{m}\cong\rho_{g',\lambda}\otimes\chi_{\ell}^{m'}$, the following hold for every prime $p$ with $p\nmid\ell N$:
    \begin{align*}
        a_{g}(p)p^{m}&=\mathrm{tr}((\rho_{g,\lambda}\otimes\chi_{\ell}^{m})(\mathrm{Frob}_{p}))=\mathrm{tr}((\rho_{g',\lambda}\otimes\chi_{\ell}^{m'})(\mathrm{Frob}_{p}))=a_{g'}(p)p^{m'},\\
        p^{k-1+2m}&=\det(p^{m}\cdot\rho_{g,\lambda}(\mathrm{Frob}_{p}))=\det((\rho_{g,\lambda}\otimes\chi_{\ell}^{m})(\mathrm{Frob}_{p}))\\
        &=\det((\rho_{g',\lambda}\otimes\chi_{\ell}^{m'})(\mathrm{Frob}_{p}))=\det(p^{m'}\cdot\rho_{g',\lambda}(\mathrm{Frob}_{p}))=p^{k'-1+2m'}.
    \end{align*}
    Thus, for all primes $p\nmid\ell N$,
    \begin{equation}\label{eq-trdetrel}
        a_{g}(p)p^{m}=a_{g'}(p)p^{m'}\quad\text{and}\quad m'-m=\frac{k-k'}{2}.
    \end{equation}

    Let $\pi_g=\bigotimes' \pi_{g,p}$ and $\pi_{g'}=\bigotimes' \pi_{g',p}$ be the cuspidal automorphic representations of $\mathrm{GL}_2(\mathbb{A}_{\mathbb{Q}})$ attached to $g$ and $g'$, respectively (cf. \cite[Section 6.7]{GH24}). Consider the representations
    $$\Pi_{g}:=\pi_{g}\otimes|\det|^{-(k-1)/2},\quad\Pi_{g'}:=\pi_{g'}\otimes|\det|^{-(k'-1)/2}.$$
    It is known that, for each prime $p\nmid N$, the local components $\pi_{g,p}$ and $\pi_{g',p}$ are unramified principal series representations of $\mathrm{GL}_{2}(\mathbb{Q}_{p})$, with Satake parameters given by the roots of
    $$x^{2}-a_{g}(p)x+p^{k-1}\quad\text{and}\quad x^{2}-a_{g'}(p)x+p^{k'-1},$$
    respectively. Hence, for every prime $p\nmid N$, the representations $\Pi_{g,p}$ and $\Pi_{g',p}$ are unramified principal series and their Satake parameters are the roots of 
    $$x^{2}-a_{g}(p)p^{-(k-1)/2}x+1\quad\text{and}\quad x^{2}-a_{g'}(p)p^{-(k'-1)/2}x+1,$$ 
    respectively. By \eqref{eq-trdetrel}, for all primes $p\nmid\ell N$,
    \begin{equation*}
        \frac{a_{g}(p)}{p^{(k-1)/2}}=\frac{a_{g'}(p)}{p^{(k'-1)/2}}.
    \end{equation*}
    Consequently, for any prime $p$ with $p\nmid\ell N$, the representations $\Pi_{g,p}$ and $\Pi_{g',p}$ have the same Satake parameters. It follows that
    $$\Pi_{g,p}\cong\Pi_{g',p}$$
    for all primes $p\nmid\ell N$. By the strong multiplicity one theorem for $\mathrm{GL}_{2}$ (cf. \cite[Theorem 11.7.2]{GH24}),
    $$\Pi_{g}\cong\Pi_{g'}.$$
    In particular,
    $$\Pi_{g,\infty}\cong\Pi_{g',\infty}$$
    as $(\mathfrak{gl}_{2}(\mathbb{R}),\mathrm{SO}_{2}(\mathbb{R}))$-module. Since $|\det|_{\infty}^{-(k-1)/2}$ and $|\det|_{\infty}^{-(k'-1)/2}$ are trivial as representations of $\mathrm{SO}_{2}(\mathbb{R})$,
    $$\Pi_{g,\infty}\cong\pi_{g,\infty}\quad\text{and}\quad\Pi_{g',\infty}\cong\pi_{g',\infty}$$
    as representations of $\mathrm{SO}_{2}(\mathbb{R})$. This implies that $\Pi_{g,\infty}$ and $\pi_{g,\infty}$ (resp. $\Pi_{g',\infty}$ and $\pi_{g',\infty}$) have the same $\mathrm{SO}_{2}(\mathbb{R})$-types. Since $\Pi_{g,\infty}\cong\Pi_{g',\infty}$, $\pi_{g,\infty}$ and $\pi_{g',\infty}$ have the same set of $\mathrm{SO}_{2}(\mathbb{R})$ types, and so their minimal $\mathrm{SO}_{2}(\mathbb{R})$-types coincide. For a newform, this minimal $\mathrm{SO}_{2}(\mathbb{R})$-type is precisely the weight, and hence $k=k'$. Returning to \eqref{eq-trdetrel}, we conclude that 
    $$m=m'.$$
    By \eqref{eq-trdetrel} again, $a_{g}(p)=a_{g'}(p)$ for all primes $p\nmid\ell N$. Therefore,
    $$g=g'$$
    by the strong multiplicity one theorem applied to newforms $g$ and $g'$.
\end{proof}

Finally, we introduce the linear independence of characters, which is the key technical result in our proof of Theorem \ref{thm1}.

\begin{lemma}\label{lem-charindep}
    Let $S$ be a finite set of primes, and let $F$ be a field of characteristic $0$.
    Let $\sigma_{1},\ldots,\sigma_{m}$ be pairwise non-isomorphic, continuous, semisimple representations which are unramified outside $S$.
    Assume that $c_{1},\ldots,c_{m}\in F$ satisfy
    $$\sum_{i=1}^{m}c_{i}\mathrm{tr}(\sigma_{i}(\mathrm{Frob}_{p}))=0$$
    for all primes $p\not\in S$.
    Then $c_{1}=\cdots=c_{m}=0$.
\end{lemma}

\begin{proof}
    See \cite[Lemma 4.1.18]{GW09}.
\end{proof}

We are now ready to prove Theorem \ref{thm1}.

\begin{proof}[Proof of Theorem \ref{thm1}]
    Let $f=\sum_{n=0}^{\infty}a_{f}(n)q^{n}$ be a prime-detecting quasimodular form on $\Gamma_{0}(N)$. 
    Then $a_{f}(p)=0$ for every prime $p\nmid N$. 
    Using the decomposition $\widetilde{M}(N)=\widetilde{E}(N)\oplus\widetilde{S}^{\mathrm{new}}(N)\oplus\widetilde{S}^{\mathrm{old}}(N)$, we can write $f$ uniquely as
    $$f=f_{E}+f_{S}^{\mathrm{new}}+f_{S}^{\mathrm{old}},$$
    where $f_{E}\in\widetilde{E}(N)$, $f_{S}^{\mathrm{new}}\in\widetilde{S}^{\mathrm{new}}(N)$ and $f_{S}^{\mathrm{old}}\in\widetilde{S}^{\mathrm{old}}(N)$.
    Set
    $$g=f-f_{S}^{\mathrm{old}}=f_{E}+f_{S}^{\mathrm{new}}.$$
    By Lemma \ref{lem-quasiold}, $a_{f_{S}^{\mathrm{old}}}(p)=0$ for every prime $p\nmid N$. 
    Thus,
    $$a_{g}(p)=a_{f}(p)-a_{f_{S}^{\mathrm{old}}}(p)=a_{f}(p)=0.$$
    Therefore, $g$ is again prime-detecting. Replacing $f$ by $g$, we may assume from now on that
    \begin{equation}\label{eq-noold}
        f=f_{E}+f_{S}^{\mathrm{new}}.
    \end{equation}

    Next, we treat the Eisenstein part $f_{E}$. 
    By \eqref{eq-qmfeis}, \eqref{eq-qmfeisbasis} and \eqref{eq-qmfeistotal}, $f_{E}$ is a finite linear combination of the constant function $1$, quasimodular forms of the form $D^{r}E_{2i}^{\varphi,t}$ ($r\in\mathbb{Z}_{\geq 0}$, $i\in\mathbb{Z}_{>0}$ and $(\varphi,t)\in A_{N,2i}$) and quasimodular forms of the form $D^{r}E_{2}$ ($r\in\mathbb{Z}_{\geq 0}$).
    In the case $(2i,\varphi)=(2,\mathbf{1}_{1})$, we have $E_{2}^{\mathbf{1}_{1},t}(\tau)=E_{2}(\tau)-tE_{2}(t\tau),$
    and hence $f_{E}$ may be rewritten as a finite linear combination of the following types of quasimodular forms: the constant function $1$; the forms $D^{r}E_{2i}^{\varphi,t}$ with $r\in\mathbb{Z}_{\geq 0}$, $i\in\mathbb{Z}_{>0}$, $(\varphi,t)\in A_{N,2i}$ and $(2i,\varphi)\neq (2,\mathbf{1}_{1})$; the forms $D^{r}E_{2}$ with $r\in\mathbb{Z}_{\geq 0}$; and the forms $D^{r}(E_{2}(t\tau))$ with $r\in\mathbb{Z}_{\geq 0}$ and $t>1$.
    By Lemma~\ref{lem-quasioldeis}(i), if $(2i,\varphi)\neq(2,\mathbf{1}_1)$, $t>1$ and $r\geq 0$, then the $p$-th Fourier coefficient of $D^{r}E_{2i}^{\varphi,t}$ is zero for every prime $p\nmid N$. 
    Moreover, the argument in the proof of Lemma~\ref{lem-quasioldeis}(ii) shows that for every $t>1$ and every $r\ge 0$, the $p$-th Fourier coefficient of $D^{r}(E_{2}(t\tau))$ is also zero. 
    Consequently, if we collect from the above linear combination of $f_{E}$ all terms involving either the constant function $1$, or a form $D^{r}E_{2i}^{\varphi,t}$ with $(2i,\varphi)\neq(2,\mathbf{1}_1)$ and $t>1$, or a form $D^{r}(E_{2}(t\tau))$ with $t>1$, and we denote their sum by $f_{E,>1}$, then $a_{f_{E,>1}}(p)=0$ for every prime $p\nmid N$.
    
    Define
    $$f_{E,\leq 1}:=f_{E}-f_{E,>1},\quad h:=f-f_{E,>1}=f_{E,\leq 1}+f_{S}^{\mathrm{new}}.$$
    Then $h\in\widetilde{M}(N)$ and, for every prime $p\nmid N$,
    \begin{equation}\label{eq-primevanishing}
        a_{h}(p)=a_{f}(p)-a_{f_{E,>1}}(p)=a_{f}(p)=0.
    \end{equation}
    In particular, $h$ is prime-detecting.
    Note that $f_{E,\leq 1}$ is a finite $\mathbb{C}$-linear combination of quasimodular forms of the form $D^{r}E_{2i}^{\varphi}$ together with quasimodular forms of the form $D^{r}E_{2}$.
    Replacing $f$ by $h$, we may assume from now on that
    \begin{equation}\label{eq-noeisold}
        f=f_{E,\leq 1}+f_{S}^{\mathrm{new}}.
    \end{equation}
    Then our claim becomes $f_{S}^{\mathrm{new}}=0$.
    By \eqref{eq-qmfeisbasis}, \eqref{eq-qmfnew}, \eqref{eq-qmfeistotal} and \eqref{eq-qmfnewtotal}, there exist a positive integer $k_{f}$ and coefficients
    \begin{align*}
        &a_{k}\in\mathbb{C}\quad(1\leq k\leq k_{f}),\quad b_{k,\varphi}\quad(1\leq k\leq k_{f},~(\varphi,1)\in A_{N,2},~\varphi\neq\mathbf{1}_{1}),\\
        &c_{k,i,\varphi}\quad(1\leq k\leq k_{f},~2\leq i\leq k,~(\varphi,1)\in A_{N,2i}),\\
        &d_{k,i,L,g}\quad(1\leq i\leq k\leq k_{f},~L\mid N,~g\in\mathcal{N}_{2i}(L))
    \end{align*}
    such that
    \begin{equation}\label{eq-finalform}
        \begin{split}
            f(\tau)&=\sum_{k=1}^{k_{f}}\sum\limits_{(\varphi,1)\in A_{N,2}}b_{k,\varphi}D^{k-1}E_{2}^{\varphi}(\tau)+\sum_{k=1}^{k_{f}}\sum_{i=2}^{k}\sum_{(\varphi,1)\in A_{N,2i}}c_{k,i,\varphi}D^{k-i}E_{2i}^{\varphi}(\tau)\\
        &\quad+\sum_{k=1}^{k_{f}}a_{k}D^{k-1}E_{2}(\tau)+\sum_{k=1}^{k_{f}}\sum_{i=1}^{k}\sum_{L\mid N}\sum_{g\in\mathcal{N}_{2i}(L)}d_{k,i,L,g}D^{k-i}g(\tau).
        \end{split}
    \end{equation}

    By Lemma \ref{lem-reduction}, we may assume without loss of generality that
    \begin{equation}\label{eq-algassumption}
        f\in \widetilde{M}(N)\cap \mathbb{Q}(\zeta_N)[\![q]\!].
    \end{equation}
    Let $K$ be a number field containing $\mathbb{Q}(\zeta_{N})$ and the Fourier coefficients of all newforms $g$ that occur in \eqref{eq-finalform}, and fix a prime $\ell\nmid N$ and a finite place $\lambda$ of $K$ lying above $\ell$.
    Since $f\in K[\![q]\!]$ and the quasimodular forms of the forms $D^{r}E_{2}^{\varphi}$, $D^{r}E_{2i}^{\varphi}$, $D^{r}E_{2}$ and $D^{r}g$ occurring in the right-hand side of \eqref{eq-finalform} are linearly independent over $K$, Lemma \ref{lem-coefffield} shows that all coefficients appearing in \eqref{eq-finalform} lie in $K$.

    For each quadruple $(k,i,L,g)$ with $1\leq i\leq k\leq k_f$, $L\mid N$, $g\in\mathcal{N}_{2i}(L)$
    and $d_{k,i,L,g}\neq 0$, put
    $$\rho_{k,i,L,g}:=\rho_{g,\lambda}\otimes\chi_{\ell}^{k-i}.$$
    Then for every prime $p\nmid N\ell$, Proposition \ref{prop-cuspgalrep} gives
    \begin{equation}\label{eq-tr-cusp}
        \mathrm{tr}\bigl(\rho_{k,i,L,g}(\mathrm{Frob}_p)\bigr)
        =\mathrm{tr}\bigl((\rho_{g,\lambda}\otimes\chi_\ell^{\,k-i})(\mathrm{Frob}_p)\bigr)
        =a_g(p)\,p^{k-i},
    \end{equation}
    which is exactly the $p$-th Fourier coefficient of $D^{k-i}g$.

    Next, consider the Eisenstein terms in \eqref{eq-finalform}. For $i\geq 2$ and $(\varphi,1)\in A_{N,2i}$ and a prime $p\nmid\ell N$,
    the $p$-th Fourier coefficient of $E_{2i}^{\varphi}$ is
    $$a_{E_{2i}^{\varphi}}(p)=2(\overline{\varphi}(p)+\varphi(p)p^{2i-1}),$$
    and therefore the $p$-th Fourier coefficient of $D^{k-i}E_{2i}^{\varphi}$ equals
    \begin{equation}\label{eq-eiscoeff}
        a_{D^{k-i}E_{2i}^{\varphi}}(p)
        =2(\overline{\varphi}(p)p^{k-i}+\varphi(p)p^{k+i-1})
        =2\mathrm{tr}((\overline{\varphi}\chi_{\ell}^{k-i}\oplus \varphi\chi_{\ell}^{k+i-1})(\mathrm{Frob}_{p})),
    \end{equation}
    where we used $\chi_{\ell}(\mathrm{Frob}_p)=p$ for $p\nmid \ell N$.
    Similarly, for $(\varphi,1)\in A_{N,2}$ with $\varphi\neq\mathbf{1}_1$ and any prime $p\nmid\ell N$,
    $$a_{D^{k-1}E_{2}^{\varphi}}(p)=2(\overline{\varphi}(p)p^{k-1}+\varphi(p)p^{k})=2\mathrm{tr}((\overline{\varphi}\chi_{\ell}^{k-1}\oplus\varphi\chi_{\ell}^{k})(\mathrm{Frob}_{p})).$$
    Finally, since $a_{E_{2}}(p)=-24(1+p)$ for all primes $p$, we have
    \begin{equation}\label{eq-etwocoeff}
        a_{D^{k-1}E_{2}}(p)=-24(p^{k-1}+p^{k})
        =-24\mathrm{tr}((\chi_{\ell}^{k-1}\oplus\chi_{\ell}^{k})(\mathrm{Frob}_{p})).
    \end{equation}

    After grouping together identical $1$-dimensional characters (if any),
    there exist finitely many pairwise distinct characters $\eta_{1},\ldots,\eta_{u}$ of $G_{\mathbb{Q}}$
    (of the form $\psi\chi_{\ell}^{r}$ with $\psi$ a Dirichlet character of conductor dividing $N$)
    and coefficients $\alpha_{1},\ldots,\alpha_{u}\in K_{\lambda}$, together with finitely many $\rho_{1},\ldots,\rho_{v}$ of the form $\rho_{g,\lambda}\otimes\chi_{\ell}^{m}$ and coefficients
    $\beta_{1},\ldots,\beta_{v}\in K_{\lambda}$, such that for every prime $p\nmid\ell N$,
    \begin{equation}\label{eq-tracerelation}
        0=a_{f}(p)=\sum_{j=1}^{u}\alpha_{j}\mathrm{tr}(\eta_{j}(\mathrm{Frob}_{p}))+\sum_{j=1}^{v}\beta_{j}\mathrm{tr}(\rho_{j}(\mathrm{Frob}_{p})).
    \end{equation} 
    By Lemma \ref{lem-rigidity}, the $2$-dimensional representations $\rho_{1},\dots,\rho_{v}$ are pairwise non-isomorphic. Since the $\eta_{j}$ are $1$-dimensional but the $\rho_{j}$ are $2$-dimensional, $\eta_{j}$ is not isomorphic to $\rho_{j'}$ for all $j,j'$. 
    Moreover, the $\eta_j$ and the $\rho_{j}$ are irreducible, and thus they are semisimple.
    By Lemma \ref{lem-charindep},
    $$\alpha_{1}=\cdots=\alpha_{u}=\beta_{1}=\cdots=\beta_{v}=0.$$
    Hence, $f_{S}^{\mathrm{new}}=0$ as desired.
\end{proof}

\section{Proof of Theorem \ref{thm: quan-1}}\label{sec-proof2}

Throughout this section, we assume that $f=\sum_{i\in I}c_iD^{m_i}f_i$,
where $I$ is a finite index set, $c_i\in\mathbb{C}^{\times}$, $m_i\geq 0$, and each $f_i$ is either an Eisenstein series or a non-CM normalized Hecke eigenform of weight $k_i$.
For each $i\in I$, let $K_{i}$ be a number field containing the Fourier coefficients of $f_i$ and $\mathcal{O}_{i}$ be a ring of integers of $K_i$.
Let $\ell$ be a prime with $(\ell,N)=1$, and let $\lambda_i$ be a prime ideal of $\mathcal{O}_i$ above $\ell$.
Let $\mathcal{O}_{i,\lambda_i}$ denote the $\lambda_i$-adic completion of $\mathcal{O}_i$, and $\varpi_i$ be a uniformizer of $\mathcal{O}_{i,\lambda_i}$.
By \eqref{eq-eiscoeff}, \eqref{eq-etwocoeff}, and Proposition~\ref{prop-cuspgalrep}, after multiplying $f_i$ by a nonzero constant, there exists a Galois representation $\rho_{i,\lambda_i}:G_{\mathbb{Q}}\to \mathrm{GL}_{2}(\mathcal{O}_{i,\lambda_i})$ such that for every prime $p\nmid N\ell$,
\[ \tr\left(\rho_{i,\lambda_i}(\Frob_p)\right)=a_{f_i}(p), \quad \det\left(\rho_{i,\lambda_i}(\Frob_p)\right)=p^{k_i-1}. \]
Let $I_1$ be the subset of $I$ consisting of $i$ such that $f_i$ is a non-CM normalized Hecke eigenform. 
To prove Theorem \ref{thm: quan-1} \textnormal{(ii)}, we recall a result on the independence of Galois representations attached to modular forms.

\begin{theorem}\label{thm4}
    For distinct $i,j\in I_1$, assume that $f_i$ is not Galois conjugate to a Dirichlet twist of $f_j$. 
    Then, for all but finitely many prime $\ell$, we have 
    \begin{equation*}
        \begin{aligned}
            &\left\{\left(\rho_{i,\lambda_i}(g)\right)_{i\in I_1} : g\in G_{\mathbb{Q}} \right\}\\
            &= \left\{(M_i)_{i\in I_1} : M_i\in \mathrm{GL}_{2}(\mathcal{O}_{i,\lambda_i}) \text{ and } \det M_i\in \left(\mathbb{Z}_{\ell}^{\times}\right)^{k_i-1} \text{ for each } i\in I_1\right\}.
        \end{aligned}
    \end{equation*}
\end{theorem}
\begin{proof}
    Loeffler \cite[Theorem 3.2.2]{L17} proved the case $n=2$ and remarked that the result should hold for general $n$. 
    The general case can be obtained by following the proof of \cite[Lemma 2.3]{CL19}.
\end{proof}

For $f=\sum_{i\in I}c_i D^{m_i}f_i$, we define $\mathcal{P}(f)$ to be the set of primes $\ell$ such that 
\begin{enumerate}
    \item[\textnormal{(i)}] for each $i\in I$, $(k_i-1,\ell(\ell-1))=1$ and $(k_i+2m_i-1,\ell(\ell-1))=1$, 
    \item[\textnormal{(ii)}] for some choice of prime ideals $(\lambda_i)_{i\in I_1}$ lying over $\ell$,   
    \begin{equation*}
        \begin{aligned}
            &\left\{\left(\rho_{i,\lambda_i}(g)\right)_{i\in I_1} : g\in G_{\mathbb{Q}} \right\}\\
            &= \left\{(M_i)_{i\in I_1} : M_i\in \mathrm{GL}_{2}(\mathcal{O}_{i,\lambda_i}) \text{ and } \det M_i\in \left(\mathbb{Z}_{\ell}^{\times}\right)^{k_i-1} \text{ for each } i\in I_1\right\}
        \end{aligned}
    \end{equation*}
    \item[\textnormal{(iii)}] for each $i\in I_1$, $\ell$ is unramified in $K_i$.
\end{enumerate}
Theorem \ref{thm4} implies that $\mathcal{P}(f)$ is nonempty. 
For non-negative integers $m$ and $n$, let $\overline{\rho_{i,\lambda_i}\otimes \chi_{\ell}^{m}}^{n}$ be the reduction of $\rho_{i,\lambda_i}\otimes \chi_{\ell}^{m}$ modulo $\varpi_i^{n}$.
Let $J$ be a subset of $I$. 
For $\boldsymbol{\lambda}:=(\lambda_i)_{i\in J}$ and $\mathbf{m}:=(m_i)_{i\in J}$, define $\overline{\rho_{J,\boldsymbol{\lambda},\mathbf{m}}}^{n}$ by 
\[ \overline{\rho_{J,\boldsymbol{\lambda},\mathbf{m}}}^{n}(g):=\left(\overline{\rho_{i,\lambda_i}\otimes \chi_{\ell}^{m_i}}^{n}(g) \right)_{i\in J}. \]
The following lemma describes the image of $\overline{\rho_{I_1,\boldsymbol{\lambda},\mathbf{m}}}^{n}$ in the case where $\ell\in \mathcal{P}(f)$.

\begin{lemma}\label{lem1}
    Assume that $\ell\in \mathcal{P}(f)$. 
    Then, we have 
    \begin{equation*}
        \begin{aligned}
            &\overline{\rho_{I_1,\boldsymbol{\lambda},\mathbf{m}}}^{n}\left(G_{\mathbb{Q}}\right)\\
            &=\left\{(M_i)_{i\in I_1}\in \prod_{i\in I_1}\mathrm{GL}_{2}(\mathcal{O}_{i,\lambda_i}/\varpi_i^{n}) : \det M_i=\overline{u}^{k_i+2m_i-1} \text{ for some } \overline{u}\in \left(\mathbb{Z}/\ell^{n}\mathbb{Z}\right)^{\times}\right\}.
        \end{aligned}
    \end{equation*}

\end{lemma}
\begin{proof}
    By Proposition \ref{prop-cuspgalrep}, it is enough to show that $\overline{\rho_{I_1,\boldsymbol{\lambda},\mathbf{m}}}^{n}$ contains
    \begin{equation*}
        \begin{aligned}
        \left\{(M_i)_{i\in I_1}\in \prod_{i\in I_1}\mathrm{GL}_{2}(\mathcal{O}_{i,\lambda_i}/\varpi_i^{n}) : \det M_i=\overline{u}^{k_i+2m_i-1} \text{ for some } \overline{u}\in \left(\mathbb{Z}/\ell^{n}\mathbb{Z}\right)^{\times}\right\}.
        \end{aligned}
    \end{equation*}
    Assume that $M_i\in \mathrm{GL}_{2}(\mathcal{O}_{i,\lambda_i}/\varpi_i^{n})$ with $\det M_i=\overline{u}^{k_i+2m_i-1}$ for some $\overline{u}\in (\mathbb{Z}/\ell^{n}\mathbb{Z})^{\times}$.
    Then, $\overline{u}^{-m_i}M_i\in \mathrm{GL}_{2}(\mathcal{O}_{i,\lambda_i}/\varpi_i^{n})$ and $\det \left(\overline{u}^{-m_i}M_i\right) = \overline{u}^{k_i-1}$.
    Since $\ell\in \mathcal{P}(f)$, it follows that there is a prime $p$ such that $\overline{u}^{-m_i}M_i=\overline{\rho_{i,\lambda_i}}^{n}(\Frob_p)$.
    Moreover, since $\ell\in \mathcal{P}(f)$ and $(\ell(\ell-1),k_i-1)=1$, we have $\overline{u}=\overline{p}$ in $\mathbb{Z}/\ell^{n}\mathbb{Z}$, where $\overline{p}$ is the reduction of $p$ modulo $\ell^{n}$.
    It follows that
    \[ \overline{\rho_{i,\lambda_i}\otimes\chi_{\ell}^{m_i}}^{n}(\Frob_p)=M_i.\]
    Therefore, we have
    \[\overline{\rho_{I_1,\boldsymbol{\lambda},\mathbf{m}}}^{n}(\Frob_p)=(M_i)_{i\in I_1}. \]
\end{proof}

Let $S$ be a set, and suppose that there exists an injection $\iota : S \hookrightarrow \mathbb{Z}_{\ell}^{m}$.
For a positive integer $n$, let $\pi_n:\mathbb{Z}_{\ell}^{m}\to \left(\mathbb{Z}_{\ell}/\ell^{n}\mathbb{Z}_{\ell}\right)^{m}$ be the canonical projection.
We say that $\dim_{M}S\leq d$ if 
\[ \# \pi_{n}\left(\iota(S)\right)=O\left(\ell^{nd}\right) \text{ as } n\to \infty. \]
Note that if $S$ is a $\ell$-adic Lie group of dimension $d$, then 
\[  \dim_{M}S\leq d. \]
To prove Theorem \ref{thm: quan-1} \textnormal{(ii)}, we use the following theorem of Serre \cite[Theorem 10]{S}, which may be viewed as Chebotarev density theorem for infinite Galois extensions with compact $\ell$-adic Lie Galois group.

\begin{theorem}\label{thm5}
    Let $K$ be a number field and let $E/K$ be a Galois extension, with Galois group $G=\mathrm{Gal}(E/K)$.
    Assume that $G$ is a compact $\ell$-adic Lie group of dimension $N$ and that $C$ is a subset of $G$ stable under conjugation with $\dim_{M}C\leq d<N$.
    Then, we get 
    \[ \# \{\mathfrak{p} : N(\mathfrak{p})\leq X \text{ and } \Frob_{\mathfrak{p}}\in C \}=O\left(\frac{X/\log X}{\epsilon(X)^{(N-d)/N}}\right), \]
    where $N(\mathfrak{p})$ denotes the absolute norm of a prime ideal $\mathfrak{p}$ of $K$ and 
    \[\epsilon(X):=(\log X)(\log \log X)^{-2}(\log \log \log X)^{-1}. \]
\end{theorem}

We first prove Theorem \ref{thm: quan-1} \textnormal{(ii)} and then prove Theorem \ref{thm: quan-1} \textnormal{(i)}.

\begin{proof}[Proof of Theorem \ref{thm: quan-1} \textnormal{(ii)}]
    Assume that the Fourier coefficients of $f$ are algebraic. 
    Let $L$ be a number field containing all the constants $c_i$ for $i\in I$.  
    Fix a prime $\ell\in\mathcal{P}(f)$ such that each $c_i$ maps to a unit in $\mathcal{O}_L\otimes {\mathbb{Z}_{\ell}}$.
    Note that for each prime $p$ with $p\nmid N\ell$, we have 
    \[ a_{D^{m_i}f_i}(p)=\tr\left(\left(\rho_{i,\lambda_i}\otimes \chi_{\ell}^{m_i}\right)(\Frob_p)\right). \]
    Hence, for any prime $p\nmid N\ell$, we have $a_f(p)=0$ if and only if
    \begin{equation}\label{eq3}
        \sum_{i\in I} c_i \tr\left(\left(\rho_{i,\lambda_i}\otimes \chi_{\ell}^{m_i}\right)(\Frob_p) \right)=0.
    \end{equation}

    Let $G$ be the image of $G_{\mathbb{Q}}$ under $\left(\rho_{i,\lambda_i}\otimes \chi_{\ell}^{m_i}\right)_{i\in I}$. 
    Since $G_{\mathbb{Q}}$ is compact, it follows that $G$ is a compact $\ell$-adic Lie group. 
    Let $N_{G}$ be a dimension of $G$. 
    Let $S$ be a subset of $G$ consisting of $(M_{i})_{i\in I}$ such that 
    \[ \sum_{i\in I} c_i \tr(M_i)=0. \]
    By \eqref{eq3}, for any prime $p\nmid N\ell$, we have $a_f(p)=0$ if and only if $\left(\left(\rho_{i,\lambda_i}\otimes\chi_{\ell}^{m_i}\right)(\Frob_p)\right)_{i\in I}\in S$. Hence, we get
    \[ \{p:p\nmid N\ell \text{ and } \left(\left(\rho_{i,\lambda_i}\otimes\chi_{\ell}^{m_i}\right)(\Frob_p)\right)_{i\in I}\in S \} =  \{p : p\nmid N\ell \text{ and } a_f(p)=0\}.\]
    
    Let $\boldsymbol{\lambda}:=(\lambda_i)_{i\in I}$ and $\mathbf{m}:=(m_i)_{i\in I}$.
    For a positive integer $n$, let 
    \[\pi_n:\prod_{i\in I} \mathrm{GL}_{2}(\mathcal{O}_{i,\lambda_i})\to \prod_{i\in I}\mathrm{GL}_{2}(\mathcal{O}_{i,\lambda_i}/\varpi_i^{n})\]
    be the canonical projection. 
    For an integer $a$ with $1\leq a<\ell^{n}$ and $(a,\ell)=1$, we define $G_{a}^{n}$ by the image under $\overline{\rho_{I,\boldsymbol{\lambda}, \mathbf{m}}}^{n}$ of the set of Frobenius elements $\Frob_p$ for primes $p\nmid N\ell$ satisfying $p\equiv a\pmod{\ell^{n}}$.
    Since $(k_i+2m_i-1,\ell(\ell-1))=1$ for each $i\in I$, we obtain a disjoint decomposition
    \[ \pi_n(G)=\bigsqcup_{a} G_{a}^{n}, \]
    where the union runs over all $a$ with $1\leq a<\ell^{n}$ and $(a,\ell)=1$.
    
    Note that if $f_i$ is an Eisenstein series, then the trace of $\rho_{i,\lambda_i}$ is the sum of two characters. 
    Consequently, for each $(M_i)_{i\in I}\in G_{a}^{n}$, the component $M_i$ is determined uniquely for $i\in I\setminus I_1$. 
    By Lemma \ref{lem1}, it follows that $G_a^{n}$ consists of those $(M_i)_{i\in I}$ such that the components $M_i$ are fixed for $i\in I\setminus I_1$ and $\det M_i=\overline{a}^{k_i+2m_i-1}$ for $i\in I_1$. 
    Here, $\overline{a}$ denotes the reduction of $a$ modulo $\ell^{n}$.
    Since $I_1$ is nonempty, we fix $j\in I_1$ with $c_j\neq 0$.
    If $(M_i)_{i\in I}\in \pi_n(S)$, then 
    \[ \sum_{i\in I}c_i\tr(M_i) = \sum_{i\in I\setminus\{j\}} c_i \tr(M_i) + c_j \tr (M_j)=0. \]
    It follows that  
    \[ \tr(M_j)=-c_{j}^{-1}\sum_{i\in I\setminus\{j\}} c_i \tr(M_i). \]
    By a direct computation (see \cite[Lemma~3.4]{CL23}), for any $x,y\in \mathcal{O}_{j,\lambda_j}/\varpi_{j}^{n}$, we have
    \begin{equation*}
    \begin{aligned}
         \#&\left\{M_j\in \mathrm{GL}_{2}(\mathcal{O}_{j,\lambda_j}/\varpi_j^{n}) : \tr M_j =x \text{ and } \det M_j=y  \right\}\\
        &\ll \frac{1}{\# (\mathcal{O}_{j,\lambda_j}/\varpi_j^{n})} \#\left\{M_j\in \mathrm{GL}_{2}(\mathcal{O}_{j,\lambda_j}/\varpi_j^{n}) : \det M_j=y  \right\}.
    \end{aligned}
    \end{equation*}
    Consequently, for any $a$ with $1\leq a < \ell^{n}$ and $(a,\ell)=1$, we get
\[ \# \left(\pi_n(S)\cap G_a^{n}\right) \ll  \frac{\# G_{a}^{n}}{\# (\mathcal{O}_{j,\lambda_j}/\varpi_j^{n})}. \]
Summing over all such $a$ with $1\leq a < \ell^{n}$ and $(a,\ell)=1$, we obtain
\[ \#\pi_n(S)\ll \frac{\# \pi_n(G)}{\# (\mathcal{O}_{j,\lambda_j}/\varpi_j^{n})}. \]
Since $\#(\mathcal{O}_{j,\lambda_j}/\varpi_j^{n})=\ell^{nf_j}$, we deduce that
\[ \dim_{M} S\leq N_{G}-f_j, \]
where $f_j$ is the residue degree of $\lambda_j$. 
By Theorem \ref{thm5}, we get
\[ \#\{p\leq X : \Frob_p\in S\} = O\left(\frac{X/\log X}{\epsilon(X)^{\alpha}}\right), \]
where $\alpha:=f_j/N$.

To complete the proof of Theorem \ref{thm: quan-1} \textnormal{(ii)}, it remains to treat the case where the Fourier coefficients of $f$ are not necessarily algebraic. 
    Let $E$ be the $\overline{\mathbb{Q}}$-vector space spanned by $\{c_i : i\in I\}$ and choose a $\overline{\mathbb{Q}}$-basis $\{e_1,\dots, e_s\}$ of $E$.
    Then, we can write
    \[ f=\sum_{r=1}^{s} e_r g_r \]
    where each $g_r$ has algebraic Fourier coefficients. 
    Since $\{e_1,\dots, e_s\}$ is linearly independent over $\overline{\mathbb{Q}}$, it follows that 
    \[ \{p : a_f(p)=0\} = \cap_{i=1}^{s} \{p : a_{g_i}(p)=0  \}. \]
    Therefore, we complete the proof of Theorem \ref{thm: quan-1} \textnormal{(ii)}.
\end{proof}

Now, we prove Theorem \ref{thm: quan-1} \textnormal{(i)}.

\begin{proof}[Proof of Theorem \ref{thm: quan-1} \textnormal{(i)}]
In the proof of Theorem \ref{thm1}, a quasimodular form $f$ can be expressed as 
\[ f=f_{E,\leq 1} + f_{S}^{\mathrm{new}} + g, \]
where $g$ satisfies $a_g(p)=0$ for all primes $p$ with $p\nmid N$. 
Thus, we may reduce to the case 
\[ f=f_{E,\leq 1} + f_{S}^{\mathrm{new}}. \]
By \eqref{eq-finalform}, we can express $f$ as $f=\sum_{i\in I}c_if_i$, where $f_i$ is an iterated derivatives of an Eisenstein series or of a Hecke eigenform. 
 Assume further that, for every $i\in I$, the Fourier coefficients of $f_i$ are algebraic integers.
 
First, assume that all Fourier coefficients of $f$ are algebraic.
 Let $K$ be a number field containing $c_i$ and the Fourier coefficients of $f_i$ for $i\in I$.
 After multiplying $f$ by a suitable nonzero constant, we may assume that the Fourier coefficients of $f$ are in $\mathcal{O}_{K}$.
 By assumption, there is a prime $p_0$ with $p_0\nmid N$ such that $a_f(p_0)\neq 0$.
 Thus, there is a prime ideal $\lambda$ of $\mathcal{O}_{K}$ such that $a_f(p_0)\not\equiv 0 \pmod{\lambda}$.
 Let $\mathcal{O}_{\lambda}$ be the $\lambda$-adic completion of $\mathcal{O}_{K}$, and let $\varpi$ be a uniformizer of $\mathcal{O}_{\lambda}$. Let $v_{\lambda}$ be the $\lambda$-adic valuation on $K$, and define
     \[ n:=\max\{1, 1-\min_{i\in I} v_{\lambda}(c_i)  \}. \]
 By \eqref{eq-tr-cusp}, \eqref{eq-eiscoeff}, and \eqref{eq-etwocoeff}, for each $i\in I$, there is a Galois representation $\rho_i : G_{\mathbb{Q}}\to \mathrm{GL}_{2}(\mathcal{O}_{\lambda})$ such that for all $p\nmid N\ell$
 \[ a_{f_i}(p)=\tr\left(\rho_{i}(\Frob_p)\right), \]
  where $(\ell)=\lambda\cap \mathbb{Z}$.
    For each $i\in I$, we consider the reduction $\overline{\rho}_{i}: G_{\mathbb{Q}}\to \mathrm{GL}_{2}(\mathcal{O}_{\lambda}/\varpi^{n})$ of $\rho_i$ modulo $\varpi^{n}$.
    Then, for each prime $p$ with $p\nmid N\ell$, we have
     \[ \tr \left(\overline{\rho_i}(\Frob_p)\right)\equiv a_{f_i}(p) \pmod{\varpi^{n}}. \]
     
     Let $L_i$ be the fixed field of the kernel of $\overline{\rho_i}$ and $L$ be the compositum of the fields $\{L_i\}$ for $i\in I$.
     Then, $L$ is a finite extension of $\mathbb{Q}$ since $L_i$ is a finite extension of $\mathbb{Q}$ for each $i$.
     Hence, if $p$ satisfies $p\nmid N\ell$ and $[\Frob_{p}]=[\Frob_{p_0}]$ in $\mathrm{Gal}(L/\mathbb{Q})$, then 
     \[ a_{f}(p_0)\equiv a_{f}(p) \pmod{\varpi^n}. \]
     It follows that 
     \begin{equation}\label{eq2}
        a_{f}(p)\neq 0. 
     \end{equation}
     By the Chebotarev density theorem, a positive proportion of primes $p$ satisfy \eqref{eq2}. 

    To complete the proof, it remains to remove the hypothesis that the Fourier coefficients of $f$ are algebraic. This follows from the same argument as in the proof of Theorem \ref{thm: quan-1} \textnormal{(ii)}.
\end{proof}

\section*{Acknowledgement}

The authors thank Dohoon Choi for pointing out that Theorem \ref{thm1} can be obtained from the independence of characters for a Galois group.

\end{document}